\documentclass[a4paper,10pt,]{article}
\usepackage{}
\usepackage{hyperref}
\usepackage{mathrsfs}
\usepackage{amsthm}
\usepackage{mathrsfs}
\usepackage{amsfonts}
\usepackage{amsmath}
\usepackage{amssymb}
\usepackage{amsthm}
\usepackage{enumerate}
\usepackage[mathscr]{eucal}
\usepackage{eqlist}
\usepackage{graphicx}

\usepackage[body={16.0true cm, 23.2true cm}]{geometry}

\newtheorem{Theorem}{Theorem}[section]

\newtheorem{Lemma}[Theorem]{Lemma}
\newtheorem{Corollary}[Theorem]{Corollary}
\newtheorem{Remark}[Theorem]{Remark}

\numberwithin{equation}{section} \allowdisplaybreaks
\usepackage{mathdots}

\renewcommand\abstract{{\bf Abstract}}

\begin{document}
\title{Critical points of solutions for mean curvature equation in strictly convex and nonconvex domains\footnote{\footnotesize The work is supported by National Natural Science Foundation of China (No.11401307, No.11401310), High level talent research fund of Nanjing Forestry University (G2014022) and Postgraduate Research \& Practice Innovation Program of Jiangsu Province (KYCX17\_0321). The second author is sponsored by Qing Lan Project of Jiangsu Province.}}

\author{Haiyun Deng$^{1}$\footnote{\footnotesize Corresponding author E-mail: haiyundengmath1989@163.com, Tel.: +86 15877935256.}, Hairong Liu$^{2}$, Long Tian$^{1}$\\[12pt]
\small \emph {$^{1}$School of Science, Nanjing University of Science and Technology, Nanjing, 210094, China;}\\
\small \emph {$^{2}$School of Science, Nanjing Forestry University, Nanjing, 210037, China}\\}
%\date{Received:  / Accepted: }
%\communicated{}
\date{}
\maketitle

\renewcommand{\labelenumi}{[\arabic{enumi}]}

\begin{abstract}{\bf:}
 In this paper, we mainly investigate the set of critical points associated to solutions of mean curvature equation with zero Dirichlet boundary condition in a strictly convex domain and a nonconvex domain respectively. Firstly, we deduce that mean curvature equation has exactly one nondegenerate critical point in a smooth, bounded and strictly convex domain of $\mathbb{R}^{n}(n\geq2)$. Secondly, we study the geometric structure about the critical set $K$ of solutions $u$ for the constant mean curvature equation in a concentric (respectively an eccentric) spherical annulus domain of $\mathbb{R}^{n}(n\geq3)$, and deduce that $K$ exists (respectively does not exist) a rotationally symmetric critical closed surface $S$. In fact, in an eccentric spherical annulus domain, $K$ is made up of finitely many isolated critical points ($p_1,p_2,\cdots,p_l$) on an axis and finitely many rotationally symmetric critical Jordan curves ($C_1,C_2,\cdots,C_k$) with respect to an axis.
\end{abstract}

{\bf Key Words:} mean curvature equation, critical point, nodal set.

{{\bf 2010 Mathematics Subject Classification.} 35B38; 35J93; 53C44. }

\section{Introduction}
~~~~In this paper we consider the following mathematical models with zero Dirichlet boundary condition
\begin{equation}\label{1.1}\begin{array}{l}
\left\{
\begin{array}{l}
Lu=f(u)~~\mbox{in}~~ \Omega,\\
u=0~~~~~~~~\mbox{on}~ \partial\Omega,
\end{array}
\right.
\end{array}\end{equation}
where $f$ is a real value function, $\Omega$ is a smooth, bounded, strictly convex and nonconvex domain in $\mathbb{R}^{n}(n\geq2)$ respectively, and $L$ is a mean curvature operator
\begin{equation}\label{1.2}\begin{array}{l}
Lu=div(\frac{\nabla u}{\sqrt{1+|\nabla u|^2}})=\sum\limits_{i,j=1}^{n}
a_{ij}(\nabla u)\frac{\partial^2 u}{\partial x_i \partial x_j},~~~n\geq 2,
\end{array}\end{equation}
where $a_{ij}=\frac{1}{\sqrt{1+|\nabla u|^2}}(\delta_{ij}-\frac{u_{x_i}u_{x_j}}{1+|\nabla u|^2})$.

Equation (\ref{1.2}) is a special case of the following $A$-Laplacian equation (see \cite{Gilbarg})
\begin{equation}\label{1.3}\begin{array}{l}
\left\{
\begin{array}{l}
\mbox{div}(A(|\nabla u|)\nabla u)=f(u)~~~\mbox{in}~\Omega,\\
u=0 ~~~\mbox{on}~ \partial\Omega,
\end{array}
\right.
\end{array}\end{equation}
which satisfies $A(h)+hA'(h)>0, h>0.$ For example, if $A(h)\equiv1,$  the  $A$-Laplacian equation is the well-known semilinear elliptic equation $\triangle u=f(u).$ On the other hand, if $A(h)=\frac{1}{\sqrt{1+h^2}},$ we obtain the well-known mean curvature equation. In this paper, we mainly consider the critical points of solutions for mean curvature equation.

Critical set of solutions for elliptic problems is a subject of important research. The investigations about critical points of solutions for elliptic equations have many results. However, the critical set $K$ has not been fully investigated, except for some few special domains, which still is an open problem for general domains, especially for higher dimension spaces. Now let us review some known results. In 1971 Makar-Limanov \cite{Makar-Limanov} solved the Poisson equation with constant inhomogeneous term in a convex domain, and proved that $u$ has one unique critical point. In 1985 Kawohl \cite{Kawohl1} extended Makar-Limanov's result under some hypothesis on the second derivative of $f.$ In 1998 Cabr\'{e} and Chanillo \cite{CabreChanillo} proved that the Poisson equation $-\triangle u=f(u)$ in smooth, bounded and convex domains of $\mathbb{R}^{n}(n\geq2)$ has exactly one nondegenerate critical point under the  assumption of semi-stable solutions. Under the similar assumptions about the domains,  Arango \cite{Arango} showed Poisson equation has exactly one nondegenerate critical point, provided that $f$ is a smooth and increasing function satisfying $f(0)>0.$ In 2008 Finn \cite{Finn} proved the uniqueness and nondegeneracy of critical points, under the same hypothesis of \cite{CabreChanillo}, and the weaker assumptions that $\Omega$ is a strictly convex $C^{2,\alpha}$ domain. Moreover, other authors have solved this problem and some other related problems in convex domains (see \cite{CaffarelliFriedman,CaffarelliSprock,ChenHuang,Deng2,Hamel,Sakaguchi}). For instance, in 2017 Deng, Liu and Tian \cite{Deng2} proved that the solution of constant mean curvature equation with Neumann or Robin boundary condition has exactly one nondegenerate critical point in $\mathbb{R}^n(n\geq 2)$. However, there exists a few results about the geometric distribution of critical set $K$ in nonconvex domains (see \cite{AlessandriniMagnanini,ArangoGomez2,Deng1}).

For nonconvex domains, the critical set $K$ of solutions for elliptic problems seems to be less considered. In 1992 Alessandrini and Magnanini \cite{AlessandriniMagnanini} studied the geometric structure of the critical set of solutions to semilinear elliptic equations in a planar nonconvex domain, whose boundary is composed of finite simple closed curves. They deduced that the critical set is made up of finitely many isolated critical points. In 2012 Arango and G\'{o}mez \cite{ArangoGomez2} considered the geometric distribution of critical points of the solutions to a quasilinear elliptic equation with Dirichlet boundary condition in strictly convex and nonconvex planar domains respectively. In 2017 Deng and Liu \cite{Deng1} investigate the geometric stucture of interior critical points of solutions $u$ to a quasilinear elliptic equation with nonhomogeneous Dirichlet boundary conditions in a simply connected or multiply connected domain $\Omega$ in $\mathbb{R}^2$. They develop a new method to prove $\Sigma_{i = 1}^k {{m_i}}+1=N$ or the different result $\Sigma_{i = 1}^k {{m_i}}=N,$ where $m_1,\cdots,m_k$ are the respective multiplicities of interior critical points $x_1,\cdots,x_k$ of $u$ and $N$ is the number of global maximum points of $u$ on $\partial\Omega$.

However, so far, the result about the critical set of solutions for quasilinear elliptic problems in higher dimension spaces is still an open problem. The goal of this paper is to obtain some results about the critical set of solutions for mean curvature equation in smooth, bounded, strictly convex and nonconvex domains of $\mathbb{R}^n(n\geq 2)$ respectively. Moreover, the domains $\Omega$ are some domains of revolution formed by taking a strictly convex and nonconvex planar domain about one axis respectively. Owing to the domains are symmetric with respect to some axis, therefore, we consider that the solutions of mean curvature equation should be symmetric about some axis, and the detail conclusion about symmetric solution of mean curvature equation can been seen in \cite{Bergner,LiNi,Serrin,Serrin2,SerrinZou}.

Throughout this paper, we shall suppose that $\Omega$ is a smooth, bounded, strictly convex and nonconvex domain respectively, and that $f$ is a real analytic, nondecreasing function. As we know that the coefficients $a_{ij}$ of mean curvature operator $L$ are analytic in $\mathbb{R}^n,$ and $L$ is uniformly elliptic. Within this assumptions and conditions, the existence of solutions cannot be taken for granted, but if there has a positive solution, then it is unique and analytic (see \cite{PucciSerrin,Serrin1}). The key idea is that we turn quasilinear elliptic equation associated to $u$ into a linear elliptic equation associated to $v=u_{\theta}.$ The results of this paper can be shown as the following three main theorems:

\vspace{0.3cm}\noindent {\bf Theorem 1.} \emph {Let $\Omega$ be a smooth, bounded and strictly convex domain of rotational symmetry with respect to an axis in $\mathbb{R}^n(n\geq3).$ Suppose that $f$ is a real analytic, nondecreasing function in $\mathbb{R}$ and that $u$ is a positive solution of equation (\ref{1.1}). Then $u$ has one unique nondegenerate critical point in $\Omega.$}\vspace{0.3cm}

\vspace{0.0cm}\noindent {\bf Theorem 2.} \emph {Let $\Omega$ be a symmetric concentric spherical annulus domain with external boundary $S^{n-1}_E$ and internal boundary $S^{n-1}_I$ in $\mathbb{R}^n(n\geq3),$ where the spherical surfaces $S^{n-1}_E$ and $S^{n-1}_I$ centered at the origin. Let $u$ be a solution of the constant mean curvature equation (\ref{1.1}) for the case $H=0$. Then the critical set $K$ of $u$ exists exactly one closed surface $S$, and $S$ is a spherical surface centered at the origin.}\vspace{0.3cm}

\vspace{0.0cm}\noindent {\bf Theorem 3.} \emph {Let $\Omega$ be a rotationally symmetric eccentric spherical annulus domain with respect to an axis in $\mathbb{R}^n(n\geq3),$ which has external boundary $S^{n-1}_E$ and internal boundary $S^{n-1}_I.$ Let $u$ be a solution of the constant mean curvature equation (\ref{1.1}) for the case $H=0$. Then the critical set $K$ of $u$ does not exist a rotationally symmetric closed surface $S$ with respect to an axis. In fact, $K$ is made up of finitely many isolated critical points ($p_1,p_2,\cdots,p_l$) on an axis and finitely many rotationally symmetric critical Jordan curves ($C_1,C_2,\cdots,C_k$) with respect to an axis.}\vspace{0.3cm}

The rest of this paper is written as follows. In Section 2, we describe the nodal set $N_{\theta}$ and the critical set $M_{\theta}$ of $u_{\theta},$ prove that $N_{\theta}$ cannot ¡°enclose¡± any subdomain of $\Omega$ and $M_{\theta}=\varnothing,$ i.e., the solution $u$ of equation (\ref{1.1}) is a Morse function. In Section 3, we give some descriptions about the geometric distribution of critical points in planar domain $\Omega,$ where $\Omega$ is a strictly convex domain and annulus domain respectively. In Section 4, our difficulty is to prove the rationality of projection. Firstly, we give some known results about symmetric solution of mean curvature equation in strictly convex domain $\Omega$ of $\mathbb{R}^n(n\geq3)$ and the detailed proof of Theorem \ref{th4.2}. Furthermore, we study the geometric distribution about the critical set $K$ of solutions $u$ for the constant mean curvature equation (\ref{1.1}) for the case $H=0$ in a concentric (respectively an eccentric) spherical annulus domain $\Omega$ of $\mathbb{R}^{n}(n\geq3)$, and deduce that $K$ exists (respectively does not exist) a rotationally symmetric critical closed surface $S$ respectively. In fact, in an eccentric spherical annulus domain $\Omega$, $K$ is made up of finitely many isolated critical points ($p_1,p_2,\cdots,p_l$) on an axis and finitely many rotationally symmetric critical Jordan curves ($C_1,C_2,\cdots,C_k$) with respect to an axis.

\section{The nodal set $N_{\theta}$ and critical set $M_{\theta}$ of $u_{\theta}$}
~~~~In order to conveniently describe the critical set $K=\{x\in\Omega|\nabla u(x)=0\},$ we need introduce some notations and auxiliary terms. Now, for any direction $\theta=(\theta_1,\theta_2,\cdots,\theta_n)\in S^{n-1}$ and any function $u,$ we define the nodal set of directional derivative $N_{\theta}=\{x\in\Omega|u_{\theta}=\nabla u(x)\cdot\theta=0\}.$ We know that $K=N_{\theta} \cap N_{\alpha},$ if $\theta$ and $\alpha$ are two noncollinear directions of $S^{n-1},$ and we consider the following critical set of directional derivative $u_{\theta}$
\begin{equation}\label{2.1}\begin{array}{l}
M_{\theta}=\{x\in N_{\theta}|\nabla u_{\theta}(x)=D^2 u(x)\cdot\theta=0\}\subset N_{\theta},
\end{array}\end{equation}
where $D^2 u(x)$ denotes the Hessian matrix of $u$ at $x.$

Near the regular points of $u_{\theta},$ $N_{\theta}$ can be locally parametrized as a solution of the following Hamiltonian system associated to $u_{\theta}$
\begin{equation}\label{2.2}\begin{array}{l}
\dot{x}(t)=B\nabla u_{\theta}(x),
\end{array}\end{equation}
where $B$ is the Poisson matrix.

In particular when $n=2,$ then
\begin{eqnarray*}
B=\left (
\begin{array}{cc}
  0 & 1 \\
  -1 & 0
\end{array}
\right),
\end{eqnarray*}
we assume $x(t)=x(t;\theta,p)$ be the solution of equation (\ref{2.2}) satisfying $x(0)=p.$ Since $f$ is a real analytic function, we know that $u$ is also analytic in $\Omega,$ so the solution of equation (\ref{2.2}) is analytic too.

Next, we will give a key idea for studying the critical set of solution $u$ for mean curvature equation (\ref{1.1}). Assume $u$ is a positive solution of equation (\ref{1.1}), for any $\theta\in S^{n-1},$ we turn quasilinear elliptic equation associated to $u$ into a linear elliptic equation associated to $v=u_{\theta}.$ Firstly, we differentiate the equation (\ref{1.1}), then take inner product with $\theta,$ hence we can get the following  equation
\begin{equation}\label{2.3}\begin{array}{l}
L_{u}v+h_1(x)\frac{\partial v}{\partial x_1}+h_2(x)\frac{\partial v}{\partial x_2}+\cdots +h_n(x)\frac{\partial v}{\partial x_n}=f'(u)v,
\end{array}\end{equation}
where
\[L_{u}v=\sum\limits_{i,j=1}^{n}
a_{ij}(\nabla u)\frac{\partial^2 v}{\partial x_i \partial x_j},~~~n\geq 2\]
and
\[h_{k}(x)=\sum\limits_{i, j=1}^{n} u_{x_i x_j}\frac{\partial a_{ij}}{\partial u_{x_k}},~~~1\leq k\leq n.\]

The following we will give the crucial lemma of the proof of Theorem \ref{th3.1} for dimension $n=2$.
\begin{Lemma}\label{le2.1}%(Lemma2.1)
Let $\Omega$ be a strictly convex planar domain. Suppose that $f$ is a real analytic, nondecreasing function in $\mathbb{R}$ and that u is a positive solution of equation (\ref{1.1}). Then for any $\theta\in S^1,$ $N_\theta$ cannot enclose any subdomain of $\Omega,$ i.e.,  $N_\theta$ without self-intersection.
\end{Lemma}
\begin{proof}[Proof] By contradiction. Assume $N_\theta$ enclose some subdomain of $\Omega$ for some $\theta,$ let $C\in N_{\theta}$ be a Jordan curve and $\Omega_{C}$ be the intersection of the interior of $C$ with $\Omega.$ Since $f$ is an nondecreasing function and $v$ satisfies the zero boundary condition. By the maximum principle, the only solution of equation (\ref{2.3}) is $v\equiv0$ in $\Omega_{C}.$  On the other hand, if $C$ is the only boundary of $\Omega_{C},$  and $u_{\theta}$ is analytic, then we have $u_{\theta}=0$ in $\Omega,$ it means that $u$ is a constant in the $\theta$ direction. Since $u=0$ on $\partial\Omega,$ we get $u=0$ in $\overline{\Omega},$ which contradicts with the fact $u>0$ in $\Omega.$ Therefore $C$ cannot be the only boundary of $\Omega_{C},$ and $\Omega_{C}$ is not simply connected. It is contradictory with the convexity of domain $\Omega,$ hence $N_{\theta}$ cannot enclose any subdomain of $\Omega.$
\end{proof}
\begin{Lemma}\label{le2.2}%(Lemma2.2)
(see\cite[Theorem 2.5]{Cheng}) If $v$ is a nonzero solution of equation (\ref{2.3}) for $n=2$. Then the critical points of $v$ on its nodal set are isolated and  the nodal set $N_{\theta}$ is a regular analytic curve at regular points. Moreover, at any critical point, the nodal set is locally an equiangular system of at least four rays splitting $\Omega$ into a finite number of connected subregions.
\end{Lemma}

\begin{Remark}\label{re2.3}%(Remark 2.3)
 According to Lemma \ref{le2.1} and Lemma \ref{le2.2}, we easily know that $M_{\theta}=\varnothing$ and $N_{\theta}$ is homeomorphic to the interval $[0,1]$, since $M_{\theta}=\varnothing,$ we get $\nabla u_{\theta}(x)=D^2u(x)\cdot \theta\neq 0$ for any $\theta,$ so $rank(D^2u(x))=2.$ Furthermore, owing to $K\subset N_{\theta}$ for any $\theta$, we deduce that the solution $u$ of equation (\ref{1.1}) is a Morse function. The descriptions about the Morse and semi-Morse function have been already studied(see \cite{Bott}).
\end{Remark}

\section{The critical set for planar domain }
~~~~In this section, we investigate the geometric distribution of critical points in planar domain $\Omega,$ where $\Omega$ is a smooth, bounded, strictly convex domain and nonconvex domain respectively.
\begin{Theorem}\label{th3.1}%(Theorem 3.1)
 Let $\Omega$ be a smooth, bounded and strictly convex domain in $\mathbb{R}^2,$ whose boundary has positive curvature. Suppose that $f$ is a real analytic, nondecreasing function in $\mathbb{R}$ and that $u$ is a positive solution of equation (\ref{1.1}). Then $u$ has one unique nondegenerate critical point in $\Omega.$
\end{Theorem}
\begin{proof}[Proof] We divide the proof into two steps.

Step 1, By Lemma \ref{le2.1} and Remark \ref{re2.3}, since $M_\theta =\varnothing$  for any $\theta  \in S^1,$ we deduce that the solution $u$ of  equation (\ref{1.1}) is a Morse function, so all critical points of  $u$ are nondegenerate.

Step 2, we will prove the uniqueness of critical points. Indeed, if $x\in \overline{\Omega}\backslash K,$ we note
\begin{equation}\label{3.1}\begin{array}{l}
\nabla u(x) = |\nabla u(x)|(\cos \lambda (x),\sin \lambda (x)),
\end{array}\end{equation}
which $\lambda(x)\in R$ and $\cos\lambda(x) = \frac{u_{x_1}}{|\nabla u|},\sin\lambda (x) = \frac{u_{x_2}}{|\nabla u|}.$ This expression defines a smooth function $\lambda$ locally in $x,$ for any
$x \in \overline \Omega  \backslash K,$ we have $x\in N_\theta$ if and only if $\theta\perp \nabla u(x),$  since $\nabla u\cdot (\cos(\lambda(x)\pm \pi /2), \sin(\lambda(x)\pm \pi /2))=0,$ hence we deduce
\[x \in {N_{\lambda (x) \pm \pi /2}}.\]
Next we need compute, locally, an expression for $\nabla\lambda.$ By formula (\ref{3.1}), since $\sin \lambda (x) = \frac{{{u_{{x_2}}}}}{{|\nabla u|}},$ then we have
\[\left\{ \begin{array}{l}
 \lambda_{x_1}\cos \lambda(x) = \frac{u_{x_1x_2}u_{x_1}^2 - u_{x_1x_1}u_{x_1}u_{x_2}}{|\nabla u|^3}, \\
 \lambda_{x_2}\cos \lambda(x) = \frac{u_{x_2x_2}u_{x_1}^2 - u_{x_1x_2}u_{x_1}u_{x_2}}{|\nabla u|^3}, \\
 \end{array} \right.\]
in case that $\cos \lambda(x)\neq 0,$ we deduce
\[\begin{array}{l}
 \nabla \lambda  = \frac{1}{|\nabla u|^2}(u_{x_1x_2}u_{x_1} - u_{x_1x_1}u_{x_2},u_{x_2x_2}u_{x_1} - u_{x_1x_2}u_{x_2}) \\
~~~~  = \frac{1}{|\nabla u|^2}{D^2}u \cdot ( - u_{x_2},u_{x_1}).
 \end{array}\]
Even though $\lambda$ be defined only  locally, the above expression allows us to define the vector field
\begin{equation}\label{3.2}\begin{array}{l}
X = \frac{1}{|\nabla u|^2}{D^2}u \cdot ( - u_{x_2},u_{x_1})~~\mbox{in}~~\overline \Omega  \backslash K,
\end{array}\end{equation}
which accords with $X(x) = \nabla \lambda (x).$

On one hand, we can know that $X(x)\neq 0$ for any $x \in \overline \Omega \backslash K.$ Indeed, since
$rank(D^2u(x))=2,$ we get $X = \frac{1}{|\nabla u|^2}D^2u \cdot ( - u_{x_2},u_{x_1}) \ne 0.$

 On the other hand, since $rank(D^2u(x))=2,$ at all points in some neighborhood of $K,$ for some constant $C>0$, we obtain
 \[|X| = \frac{1}{|\nabla u|}|{D^2}u \cdot (-\sin \lambda (x),\cos \lambda (x)) \ge \frac{C}{|\nabla u|}\]
in a neighborhood of $K.$ Since $X(x)\ne0$ in  $x \in \overline \Omega  \backslash K,$ then we deduce
\begin{equation}\label{3.3}\begin{array}{l}
|X| \ge \frac{C}{{|\nabla u|}}~~\mbox{in}~~x\in \overline{\Omega}\backslash K.
\end{array}\end{equation}

By Hopf  boundary lemma, we get $\left\langle {X,t} \right\rangle  = |\nabla u{|^{ - 1}}{u_{tt}} = \kappa,$ where $\langle \cdot,\cdot \rangle$ denotes interior product, $\kappa$ is the curvature of $\partial\Omega$ and $t$ is the positive oriented tangent vector to $\partial\Omega$. Hence we have
\[\left\langle {X,t} \right\rangle  > 0~~\mbox{on}~~\partial\Omega.\]
Next we revise $X$ in a neighborhood of $\Omega$ to agree with $t$ on $\partial\Omega.$ Further there exists a smooth vector field $Y$ on $\overline{\Omega}\backslash K$ such that
\begin{equation}\label{3.4}\begin{array}{l}
Y = \left\{ \begin{array}{l}
 X~~\mbox{in}~~\Omega \backslash K, \\
 t~~~\mbox{on} ~~\partial \Omega , \\
 \end{array} \right.
\end{array}\end{equation}
so $\langle Y,X\rangle >0$ in $\overline{\Omega}\backslash K.$

Then we define another vector field $Z$ associated to $Y,$ which is smooth in $x\in \overline{\Omega}\backslash K,$ and tangent to $\partial\Omega.$
\begin{equation}\label{3.5}\begin{array}{l}
Z= \left\{ \begin{array}{l}
 \frac{Y}{\langle Y,X\rangle}~~\mbox{in}~~\overline{\Omega}\backslash K,\\
 0~~~~~~~\mbox{in} ~~K , \\
 \end{array} \right.
\end{array}\end{equation}
 we claim that field vector $Z$ can be extended to be Lipschitz in $\overline{\Omega},$ that is
\begin{equation}\label{3.6}\begin{array}{l}
|Z({x_2}) - Z({x_1})| \le C|{x_2} - {x_1}| ~~\forall~~{x_1},{x_2} \in \overline{\Omega}
\end{array}\end{equation}
for some constant $C>0.$

Indeed, since $Z=\frac{X}{|X|^2}$ in $\overline{\Omega}\backslash K,$ by (\ref{3.3}) we have
\begin{equation}\label{3.7}\begin{array}{l}
|Z| \le C|\nabla u| ~~\mbox{in}~~ \overline{\Omega}
\end{array}\end{equation}
for another positive constant $C>0,$ hence $Z$ is continuous in $\overline{\Omega}.$

Firstly, if either $x_1$ or $x_2$ belong to $K,$ without loss of generality, let $x_1\in K,$ then
\begin{equation}\label{3.8}\begin{array}{l}
|Z(x_2)-Z(x_1)|\leq |Z(x_2)|+|Z(x_1)|=|Z(x_2)|\\
~~~~~~~~~~~~~~~~~~~~\leq C|\nabla u(x_2)|=C|\nabla u(x_2)-\nabla u(x_1)|\\
~~~~~~~~~~~~~~~~~~~~\leq C|x_2-x_1|.
\end{array}\end{equation}
Secondly, if a segment $l$ joining $x_1$ and $x_2$ intersects with $K,$ note that $l\subset \overline{\Omega},$ since we select $\overline{\Omega}$ to be convex. According to (\ref{3.8}),  for any $p\in K,$ we have
\begin{equation*}\label{}\begin{array}{l}
|Z(x_2)-Z(x_1)|\leq |Z(x_2)-Z(p)|+|Z(x_1)-Z(p)|\\
~~~~~~~~~~~~~~~~~~~~= |Z(x_2)|+|Z(x_1)|\\
~~~~~~~~~~~~~~~~~~~~\leq C|x_2-x_1|.
\end{array}\end{equation*}
Finally, suppose that this segment $l$ is included in $\overline{\Omega}\backslash K.$ Hence we can differentiate $Z$ along $l,$ where $D$ denotes the full differential, by (\ref{3.3}), we have
\[ |DZ|=|D(\frac{X}{|X|^2})|\leq C\frac{|DX|}{|X|^2}\leq C|\nabla u|^2|DX|~~\mbox{in}~~\overline{\Omega}\backslash K.\]
Note that differentiating $X=\frac{1}{|\nabla u|^2}D^2u\cdot (-u_{x_2},u_{x_1})$ in $\overline{\Omega}\backslash K,$ we get $|DX|\leq C|\nabla u|^{-2}.$ So we deduce that $|DZ|\leq C$ along $l,$ and complete the proof of (\ref{3.6}).

Since each Lipschitz vector field locally can generate a one-parameter transformation group(i.e., flow). Now, we assume that there is an open convex neighborhood $E$ of $K,$ provides $K\subset E \subset \overline{E} \subset \Omega.$ Next, we will consider the one-parameter transformation group $\varphi_t$ at time $t$ associated to the Lipschitz vector field $Z,$ which satisfies
\begin{equation}%(formula 3.9)
\begin{cases}
\varphi_0(p)=p~~~~~~~~~ \forall~ p\in E,\\
\varphi_s\circ \varphi_t=\varphi_{s+t}~~~\forall~ s,t\in \mathbb{R},
\end{cases}
\end{equation}
and
\begin{equation}\label{3.10}\begin{array}{l}
Z(p)=\frac{d\varphi_t(p)}{dt}\mid_{t=0}~~~\forall~~ p\in E.
\end{array}\end{equation}
Since $Z$ is parallel to $\partial\Omega,$ hence we know that $\varphi_t$ is a continuous flow that keep $\overline{\Omega}$ invariant. Moreover, the flow $\varphi_t$ keep any critical point of $u$ fixed. Then, we can easily to deduce that
\[\varphi_t(N_\theta)\subset N_{\theta + t}\]
for any nodal set $N_\theta$(i.e., path associated with the flow $\varphi_t$) and any time $t$. Indeed, by reversing time, we know that $\varphi_t$ is an homeomorphism from $N_\theta$ onto $N_{\theta+t},$ i.e., $\varphi_t(N_\theta)=N_{\theta+t}.$ Then we choose the time $t=\pi,$ we get
\begin{equation}\label{3.11}\begin{array}{l}
\varphi_{\pi}(N_0)=N_{0+\pi}=N_0,
\end{array}\end{equation}
where $\varphi_{\pi}$ is a homeomorphism of $N_0$ that interchanges the two end-points of $N_0.$ According to Remark \ref{re2.3}, since $N_0$ is homeomorphic to the interval $[0,1],$ so $\varphi_{\pi}$ keep that $N_0$ has one unique fixed point. Moreover, because the flow $\varphi_{\pi}$ keep any critical point of $u$ fixed, thus the solution $u$ of equation (\ref{1.1}) has exactly one critical point.
\end{proof}
\begin{Remark}\label{re3.2} %(Remark 3.2)
The convexity of all the nodal set of any solution $u$ in Theorem \ref{th3.1} is an open problem, even for semilinear case in dimension 2. Note that the strict convexity of  nodal set is a stronger property than the uniqueness of critical points of the solution $u$. The related research results can be found in \cite{Kawohl2}.
\end{Remark}

Next we study the geometric structure about critical set $K$ of solutions $u$ for the constant mean curvature equation in planar nonconvex domain $\Omega,$ where $\Omega$ is a concentric and an eccentric circle annulus domain respectively.
\begin{Lemma}\label{le3.3}%(Lemma 3.3)
 Let $\Omega$ be a symmetric planar concentric circle annulus domain with internal boundary $\gamma_I$ and external boundary $\gamma_E$, where the circles $\gamma_I$ and $\gamma_E$ centered at the origin. Let $u$ be a solution of the constant mean curvature equation (\ref{1.1}). Then $u$ has exactly one critical Jordan curve $C$ in $\Omega,$ and $C$ is a circle centered at the origin.
\end{Lemma}
\begin{proof}[Proof]
Due to the results of Theorem 11 in \cite{ArangoGomez2}, we know that the critical set $K$ of solution $u$ is either finitely many isolated critical points, or is made up of exactly one critical Jordan curve, that is, the critical points and critical Jordan curve can't exist at the same time. And according to the results of \cite{Bergner}, the constant mean curvature equation (\ref{1.1}) has a unique radial symmetric solution $u$ in concentric circle annulus domain $\Omega$, so we deduce that the  critical set $K$ of solution $u$ reduces to exactly one critical Jordan curve $C$, and $C$ is a circle centered at the origin.
\end{proof}
\begin{Remark}\label{re3.4} %(Remark 3.4)
For the case of dimension $n=2,$ if the domain $\Omega$ is a symmetric planar eccentric annulus with respect to one axis, with internal boundary $\gamma_I$ and external boundary $\gamma_E$, moreover, $\gamma_I$ and $\gamma_E$ are circles, then the constant mean curvature equation (\ref{1.1}) exists only a finite number of critical points, the result from \cite{ArangoGomez2}. At the same time, if critical set $K$ exists a Jordan curve $C$ in an annulus domain $\Omega$, then the geometry structure of $N_\theta$ as follows:

(1) For any $\theta\in S^1,$ $N_\theta$ contains the curve $C$.

(2) There exists exactly two corresponding branching points of $N_\theta$ in critical Jordan curve $C,$ denotes by points $p$ and $p^\ast,$ and $\theta$ is tangent to $C$ at this two points.

(3) The nodal set $N_\theta$ exists exactly four branches departing from $p$ (respectively $p^\ast$), where two branches of them are included in $C,$ and the other two branches end respectively at the internal boundary $\gamma_I$ and external boundary $\gamma_E.$

(4) The points on external boundary $\gamma_E$ (respectively $\gamma_I$), the branches starting at $q$ and $q^\ast$ end, are corresponding points, and $\theta$ is tangent to $\partial\Omega$ at this two points.

(5) $N_\theta$ contains no any other points except for the above (1) $\sim$ (4).
\end{Remark}
For a given $\theta\in S^1,$ next we will give the geometric structure of $N_\theta.$
\begin{center}
  \includegraphics[width=4cm,height=4cm]{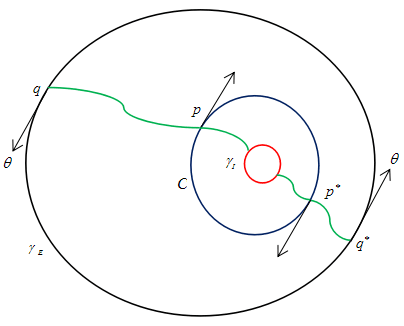}\\
  \scriptsize {\bf Figure 1}~~The geometric distribution of $N_\theta$.
\end{center}

\section{The critical set for higher dimension spaces} %(section 4)
~~~~This section is aimed to deduce that the geometric structure of critical set $K$ for mean curvature equation (\ref{1.1}) in higher dimension spaces. Next we will give some results about radial symmetric solution of mean curvature equation (\ref{1.1}) in strictly convex domains, as follows:
\begin{Lemma} \label{le4.1} %(Lemma 4.1)
(see\cite[Theorem 8.2.2]{PucciSerrin}) Let $\Omega$ be an open ball in $\mathbb{R}^n, n\geq3.$ Assume $u\in C^1(\overline{\Omega})$ is a distribution solution of the problem (\ref{1.1}), while the function $f(u)$ is locally Lipschitz continuous in $\mathbb{R}^{+}_0.$ Then every solution $u\in C^1(\overline{\Omega})$ is radially symmetric and satisfies $u_r<0$.
\end{Lemma}

 According to the above results about symmetric solution of problem (\ref{1.1}), next we investigate the geometric distribution of critical set for higher dimension spaces. Our difficulty is to prove the rationality of projection of higher dimensional space onto two dimension plane.

\begin{Theorem}\label{th4.2} %(Theorem 4.2)
 Let $\Omega$ be a smooth, bounded and strictly convex domain of rotational symmetry with respect to an axis in $\mathbb{R}^n(n\geq3).$ Suppose that $f$ is a real analytic, nondecreasing function in $\mathbb{R}$ and that $u$ is a positive solution of equation (\ref{1.1}). Then $u$ has one unique nondegenerate critical point in $\Omega.$
\end{Theorem}
\begin{proof}[Proof]
 Without loss generality, let $\Omega$ be a domain of revolution formed by taking a strictly convex planar domain in the $x_1,x_n$ plane with respect to the $x_n$ axis. In the sequel, $x=(x',x_n), x'=(x_1,\cdots,x_{n-1})$ and $r=\sqrt{x_1^2+\cdots+x_{n-1}^2}.$

Due to the results of Pucci and Serrin \cite{PucciSerrin,Serrin2}, we deduce that the solution $u$ satisfies
\begin{equation}\label{4.1}\begin{array}{l}
u(x',x_n)=u(|x'|,x_n)\triangleq v(r,x_n)
\end{array}\end{equation}
and
\begin{equation}\label{4.2}\begin{array}{l}
\frac{\partial v}{\partial r}(r,x_n)<0$ $ $ for $ $ $  r\neq0.
\end{array}\end{equation}

From (\ref{4.2}), we can know that the critical points of $u$ lie on $x_n$ axis. Next, according to (\ref{4.1}) we have that
\begin{equation}\label{4.3}\begin{array}{l}
u_{x_n}(x',x_n)=v_{x_n}(r,x_n).
\end{array}\end{equation}
Moreover, we can deduce that $u_{x_n}$ satisfies the following equation
\[(Lu)_{x_n}=\sum\limits_{i,j=1}^{n}
a_{ij}(\nabla u)\frac{\partial^2 u_{x_n}}{\partial x_i \partial x_j}+\sum\limits_{i,j=1}^{n}
\frac{\partial a_{ij}(\nabla u) }{\partial x_n }\frac{\partial^2 u}{\partial x_i \partial x_j}=f'(u)u_{x_n},~~~n\geq 2,\]
That is
\begin{equation}\label{4.4}\begin{array}{l}
\mathscr{L}u_{x_n}=\sum\limits_{i,j=1}^{n}
a_{ij}(\nabla u)\frac{\partial^2 u_{x_n}}{\partial x_i \partial x_j}+\sum\limits_{i,j=1}^{n}\frac{\partial^2 u}{\partial x_i \partial x_j}
\frac{\partial a_{ij}(\nabla u)}{\partial x_n}-f'(u)u_{x_n}=0,
\end{array}\end{equation}
where $a_{ij}=\frac{1}{\sqrt{1+|\nabla u|^2}}(\delta_{ij}-\frac{u_{x_i}u_{x_j}}{1+|\nabla u|^2})$ and $\frac{\partial a_{ij}(\nabla u)}{\partial x_n }$ as the first derivative of $u_{x_n}$.

By (\ref{1.1}), the strict convexity of $\Omega$ and the Hopf boundary point lemma, we can know that $u_{x_n}$ vanishes precisely on the $n-2$ dimensional sphere given by
\[S=\{x_n=a\}\cap\partial\Omega,\]
for some $a\in \mathbb{R}.$ For convenience, we define the nodal set
\[\mathscr{N}=\{x\in \Omega|u_{x_n}(x)=0\}.\]
It is clear that all critical points of solution $u$ are contained in $\mathscr{N}$. Also from (\ref{4.3}), $\mathscr{N}$ is rotationally invariant about the $x_n$ axis.

So we turn the mean curvature equation (\ref{1.1}) for dimension $n$
\begin{equation}\label{4.5}\begin{array}{l}
Lu=div(\frac{\nabla u}{\sqrt{1+|\nabla u|^2}})=f(u)
\end{array}\end{equation}
into the following similar mean curvature equation for dimension 2
\[Lv=div(\frac{\nabla v}{\sqrt{1+|\nabla v|^2}})+\frac{1}{\sqrt{1+|\nabla v|^2}}\frac{n-2}{r}v_r=f(v),\]
that is
\begin{equation}\label{4.6}\begin{array}{l}
Lv=\sum\limits_{i,j=1}^{2}a_{ij}(\nabla v)v_{ij}+\frac{1}{\sqrt{1+|\nabla v|^2}}\frac{n-2}{r}v_r=f(v),
\end{array}\end{equation}
where $\nabla v=(\frac{\partial v}{\partial r},\frac{\partial v}{\partial x_n}),$  $a_{ij}=\frac{1}{\sqrt{1+|\nabla v|^2}}(\delta_{ij}-\frac{v_iv_j}{1+|\nabla v|^2})$ and $v_1=\frac{\partial v}{\partial r}, v_2=\frac{\partial v}{\partial x_n}.$

 Next, the proof is essentially same as the proof of Theorem 2 in \cite{CabreChanillo}. Their work is based on the results of Gidas, Ni and Nirenberg \cite{Gidas}, the ideas of Payne \cite{Payne} and Sperb \cite{Sperb}. To prove that, on one hand, whenever critical set has exactly one point, since all critical points of $u$ are contained in $\mathscr{N}\cap\{x_2=\cdots=x_{n-1}=0\}$ and lie on the $x_n$ axis. The nodal set $\mathscr{N}=\{x\in \Omega | u_{x_n}(x)=0\}$ is rotationally invariant about the $x_n$ axis, formed by a set $N_2$ contained in the $x_1,x_n$ 2-dimensional plane rotation about the $x_n$ axis, by (\ref{4.6}), where $N_2$ can be seen as the projection of $\mathscr{N}$ in the $x_1,x_n$ 2-dimensional plane and $\mathscr{N}$ cannot enclose any subdomain of $\Omega$ (By Lemma \ref{le2.1}, $N_2$ cannot enclose any planar subdomain of $\Omega\cap\{x_2=\cdots=x_{n-1}=0\},$ where $N_2$ looks like the nodal set of some homogeneous polynomial in $x_1,x_n.$). Because $N_2$ is symmetric with respect to the $x_n$ axis and intersects the $x_n$ axis at exactly one point, hence we prove the uniqueness of critical points.

On the other hand, how to show that critical point $p$ is nondegenerate, we restatement that $u$ is rotationally symmetric with respect to $x_n$ axis and critical point $p$ lies on this axis. From (\ref{4.1}) and (\ref{4.2}), we have that $\{u_{x_k}=0\}=\{x_k=0\}\cap \Omega$ for all $1\leq k\leq n-1.$ Hence $u_{x_ix_j}(p)=0$ for any index $1\leq i<j\leq n,$ that is, $D^2u(p)$ is diagonal. By (\ref{4.2}), we can know that $u_{x_k}<0$ in domain $\mathscr{D}_k=\{x_k>0\}\cap \Omega$ for $1\leq k\leq n-1.$ What's more, in domain $\mathscr{D}_k$, $u_{x_k}$ satisfies
\begin{equation}\label{4.7}\begin{array}{l}
\mathscr{L}u_{x_k}=\sum\limits_{i,j=1}^{n}
a_{ij}(\nabla u)\frac{\partial^2 u_{x_k}}{\partial x_i \partial x_j}+\sum\limits_{i,j=1}^{n}\frac{\partial^2 u}{\partial x_i \partial x_j}
\frac{\partial a_{ij}(\nabla u)}{\partial x_k}-f'(u)u_{x_k}=0,
\end{array}\end{equation}
where $\frac{\partial a_{ij}(\nabla u)}{\partial x_k }$ as the first derivative of $u_{x_k}$. According to the Hopf boundary point lemma, we deduce that $u_{x_kx_k}(p)<0$ for all $1\leq k\leq n-1,$ where critical point $p\in \partial \mathscr{D}_k.$

Finally, we recall that the function $u_{x_n}$ satisfies (\ref{4.4}). By the definition of $\mathscr{N}$, $u_{x_n}<0$ to one side of $\mathscr{N}.$ Then applying the Hopf boundary point lemma to $u_{x_n}$ at $p\in \mathscr{N},$ we have that $u_{x_nx_n}(p)<0.$ So we prove that the Hessian matrix $D^2u(x)$ of $u$ is diagonal and negative definite at critical point $p$, hence $p$ is a unique nondegenerate critical point.
\end{proof}

Next we will study the geometric structure about critical set $K$ of solutions $u$ for the constant mean curvature equation (i.e.$f(u)=H, H=constant$) in a symmetric concentric spherical annulus domain with external boundary $S^{n-1}_E$ and internal boundary $S^{n-1}_I$ of $\mathbb{R}^n(n\geq3),$ where the spherical surfaces $S^{n-1}_E$ and $S^{n-1}_I$ centered at the origin. It is known that for $H$ small enough, there exists a unique symmetric solution $u$ for the constant mean curvature equation (\ref{1.1}), where $u$ satisfies that $\frac{\partial u}{\partial r}>0$ for the case $H=0$ (see \cite{Bergner}).

\begin{Theorem}\label{th4.3} %(Theorem 4.3)
Let $\Omega$ be a symmetric concentric spherical annulus domain with external boundary $S^{n-1}_E$ and internal boundary $S^{n-1}_I$ in $\mathbb{R}^n(n\geq3),$ where the spherical surfaces $S^{n-1}_E$ and $S^{n-1}_I$ centered at the origin. Let $u$ be a solution of the constant mean curvature equation (\ref{1.1}) for the case $H=0$. Then the critical set $K$ of $u$ exists exactly one closed surface $S$, and $S$ is a spherical surface centered at the origin.
\end{Theorem}
\begin{proof}[Proof]
Without loss of generality, we assume that the domain $\Omega$ is a domain of revolution formed by taking a symmetric planar concentric circle annulus domain $\Omega'$ in the $x_1, x_n$ plane about $x_n$ axis, where the domain $\Omega'$ centered at the origin.

By the results of Bergner \cite{Bergner}, so we deduce that the solution $u$ satisfies
\begin{equation}\label{4.8}\begin{array}{l}
u(x',x_n)=u(|x'|,x_n)\triangleq v(x_n,r)
\end{array}\end{equation}
and
\begin{equation}\label{4.9}\begin{array}{l}
\frac{\partial v}{\partial r}(x_n,r)>0,
\end{array}\end{equation}
where $x'=(x_1,\cdots,x_{n-1})$ and $r=\sqrt{x_1^2+\cdots+x_{n-1}^2}.$

So we turn the constant mean curvature equation (\ref{1.1}) for dimension $n$
\begin{equation}\label{4.10}\begin{array}{l}
Lu=div(\frac{\nabla u}{\sqrt{1+|\nabla u|^2}})=0
\end{array}\end{equation}
into the following similar mean curvature equation for dimension 2
\[Lv=div(\frac{\nabla v}{\sqrt{1+|\nabla v|^2}})+\frac{1}{\sqrt{1+|\nabla v|^2}}\frac{n-2}{r}v_r=0,\]
that is
\begin{equation}\label{4.11}\begin{array}{l}
Lv=\sum\limits_{i,j=1}^{2}a_{ij}(\nabla v)v_{ij}+\frac{1}{\sqrt{1+|\nabla v|^2}}\frac{n-2}{r}v_r=0,
\end{array}\end{equation}
where $\nabla v=(\frac{\partial v}{\partial x_n},\frac{\partial v}{\partial r}),$  $a_{ij}=\frac{1}{\sqrt{1+|\nabla v|^2}}(\delta_{ij}-\frac{v_iv_j}{1+|\nabla v|^2})$ and $v_1=\frac{\partial v}{\partial x_n}, v_2=\frac{\partial v}{\partial r}.$

For any $\theta=(\theta_1,\theta_2)=(\cos \alpha,\sin \alpha)\in S^1,$ where $\alpha \in [0,\pi).$ We turn quasilinear elliptic equation associated to $v$ into a linear elliptic equation associated to $w=v_{\theta}=\nabla v\cdot \theta.$ Firstly, we differentiate the equation (\ref{4.11}), then take inner product with $\theta.$ In order to conveniently, we denote $y=(y_1,y_2)=(x_n,r),$ hence we can get the following equation
\begin{equation}\label{4.12}\begin{array}{l}
L_{v}w+h_1(y)\frac{\partial w}{\partial y_1}+h_2(y)\frac{\partial w}{\partial y_2}+\frac{1}{(1+|\nabla v|^2)^{\frac{3}{2}}}\frac{n-2}{r}[(1+v_{y_1}^2)\frac{\partial w}{\partial y_2}-v_{y_1}v_{y_2}\frac{\partial w}{\partial y_1}] =\frac{1}{\sqrt{1+|\nabla v|^2}} \frac{n-2}{r^2}v_r \theta_2,
\end{array}\end{equation}
where
\[L_{v}w=\sum\limits_{i,j=1}^{2}
a_{ij}(\nabla v)\frac{\partial^2 w}{\partial y_i \partial y_j}\]
and
\[h_{k}(y)=\sum\limits_{i, j=1}^{2} v_{y_i y_j}\frac{\partial a_{ij}}{\partial v_{y_k}},~~~1\leq k\leq 2.\]
By (\ref{4.9}) and (\ref{4.12}), we deduce that
\begin{equation}\label{4.13}\begin{array}{l}
Lw=L_{v}w+h_1(y)\frac{\partial w}{\partial y_1}+h_2(y)\frac{\partial w}{\partial y_2}+\frac{1}{(1+|\nabla v|^2)^{\frac{3}{2}}}\frac{n-2}{r}[(1+v_{y_1}^2)\frac{\partial w}{\partial y_2}-v_{y_1}v_{y_2}\frac{\partial w}{\partial y_1}]\geq 0.
\end{array}\end{equation}
By (\ref{4.13}), so we can consider the result that graphic projected onto $x_1, x_n$ plane. Due to Lemma \ref{le3.3}, we have known that the constant mean curvature equation (\ref{1.1}) exists only one unique critical Jordan curve $C$ in symmetric planar concentric circle annulus domain. In turn, we rotate the geometry distribution of critical Jordan curve C in symmetric planar concentric circle annulus domain with respect to $x_n$ axis. Therefore we get the geometry distribution of critical set $K$ in a symmetric concentric spherical annulus domain $\Omega$ as shown in the following Figure 2.
\begin{center}
  \includegraphics[width=11.0cm,height=5.0cm]{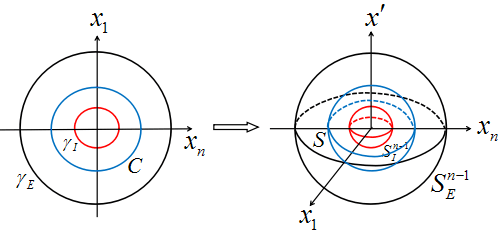}\\
  \scriptsize {\bf Figure 2} ~~The geometric distribution of critical set $K$ in a symmetric concentric spherical annulus domain.
\end{center}

Figure 2 shows that the critical set $K$ of $u$ exists exactly one closed surface $S$, and $S$ is a spherical surface centered at the origin.
\end{proof}

The rest of this section is aimed to prove that, in the case of the constant mean curvature equation in a rotationally symmetric eccentric spherical annulus domain $\Omega$ with respect to an axis, the critical set $K$ does not exist a critical closed surface $S$, where $S$ is rotationally symmetric with respect to an axis.
\begin{Theorem}\label{th4.4} %(Theorem 4.4)
 Let $\Omega$ be a rotationally symmetric eccentric spherical annulus domain with respect to an axis in $\mathbb{R}^n(n\geq3),$ which has external boundary $S^{n-1}_E$ and internal boundary $S^{n-1}_I.$ Let $u$ be a solution of the constant mean curvature equation (\ref{1.1}) for the case $H=0$. Then the critical set $K$ of $u$ does not exist a closed surface $S$, where $S$ is rotationally symmetric with respect to an axis.
\end{Theorem}
\begin{proof}[Proof]
 The proof is based on the idea of Theorem \ref{th4.3}, and we prove the theorem for two cases. Without loss of generality, we assume that the center of spherical surfaces $S^{n-1}_E$ and $S^{n-1}_I$ both on $x_n$ axis, so we deduce that the solution $u$ satisfies
\[u(x',x_n)=u(|x'|,x_n)\]
where $x'=(x_1,\cdots,x_{n-1}).$

Case 1, if the critical set $K$ enclose a subdomain of $\Omega,$ denotes by closed surface $S$, where $S$ is rotationally symmetric with respect to $x_n$ axis. According to the assumptions, we can know that critical set $K$ is the closed surface $S$ and the center of $S$ on $x_n$ axis. Because the domain $\Omega$ and the solution $u$ are rotationally symmetric with respect to $x_n$ axis, in the same way, so we can consider the result that graphic projected onto a two-dimensional plane which pass through $x_n$ axis, without loss of generality, denotes by $x_1, x_n$ plane. Therefore we get the geometry distribution of critical Jordan curve $C$ in planar nonconvex domain as shown in the following Figure 3.
\begin{center}
  \includegraphics[width=12cm,height=4.5cm]{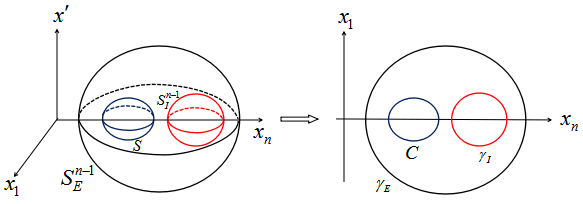}\\
  \scriptsize {\bf Figure 3} ~~The geometric distribution of critical set $K$ for case 1.
\end{center}

Because the interior of critical Jordan curve $C$ is simply connected, it is contradictory with Lemma \ref{le2.1}, so the critical set $K$ cannot enclose a subdomain of $\Omega.$

Case 2, if the critical set $K$ enclose the internal boundary $S^{n-1}_I,$ denotes by closed surface $S$, where $S$ is rotationally symmetric with respect to $x_n$ axis. we can consider the result that graphic projected onto $x_1, x_n$ plane, hence we can get the following geometry distribution of critical Jordan curve $C$ in planar nonconvex domain.
\begin{center}
  \includegraphics[width=11cm,height=3.8cm]{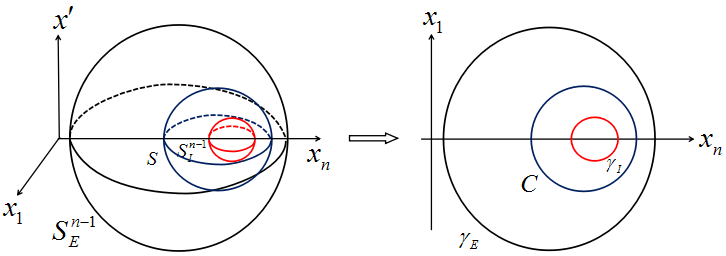}\\
  \scriptsize {\bf Figure 4} ~~The geometric distribution of critical set $K$ for case 2.
\end{center}

Due to Remark \ref{re3.4}, we have known that constant mean curvature equation (\ref{1.1}) exists only a finite number of isolated critical points in a symmetric planar eccentric circle annulus domain $\Omega$ with respect to one axis. Figure 4 shows that it is contradictory with Remark \ref{re3.4}, Therefore $K$ cannot enclose the internal boundary $S^{n-1}_I.$
\end{proof}
As an incidental consequence of Theorem \ref{th4.4} we can fully describe the geometric distribution of critical set $K$ to the constant mean curvature equation (\ref{1.1}) in a rotationally symmetric eccentric spherical annulus domain $\Omega$ with respect to an axis in $\mathbb{R}^n(n\geq3),$ as stated in the following corollary.

\begin{Corollary}\label{Co4.5} %(Corollary 4.5)
 We have known that the constant mean curvature equation (\ref{1.1}) exists only a finite number of isolated critical points in a planar eccentric circle annulus domain $\Omega$, where $\Omega$ is symmetric with respect to one axis. Hence we can deduce that the geometric distribution of critical set $K$ of the constant mean curvature equation (\ref{1.1}) for the case $H=0$ in a rotationally symmetric eccentric spherical annulus domain $\Omega$ with respect to an axis in $\mathbb{R}^n(n\geq3),$ without loss of generality, denotes by $x_n$ axis. Then critical set $K$ is made up of finitely many isolated critical points ($p_1,p_2,\cdots,p_l$) on $x_n$ axis and finitely many rotationally symmetric critical Jordan curves ($C_1,C_2,\cdots,C_k$) with respect to $x_n$ axis, and the geometric distribution of critical set $K$ as shown in the following Figure 5.
\end{Corollary}
\begin{center}
  \includegraphics[width=11cm,height=4.5cm]{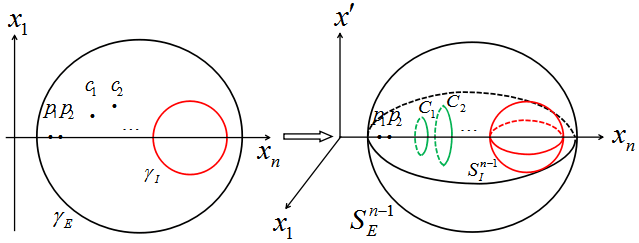}\\
 \scriptsize {\bf Figure 5} ~~The geometric distribution of critical set $K$ in an eccentric spherical annulus.
\end{center}

\begin{Remark}\label{re4.6} %(Remark 4.6)
If the general quasilinear elliptic equation $Lu=\sum\limits_{i,j=1}^{n}a_{ij}(\nabla u)\frac{\partial^2 u}{\partial x_i\partial x_j}=f(u)$ satisfies the above corresponding assumptions, conditions and exists the corresponding symmetric solution as in Theorem \ref{th4.2}, Theorem \ref{th4.3} and Theorem \ref{th4.4}, then the general quasilinear elliptic equation has the same results about the geometric distribution of critical set $K$ as in Theorem \ref{th4.2}, Theorem \ref{th4.3} and Theorem \ref{th4.4}.
\end{Remark}

\noindent  \textbf{Acknowledgement.} The first author is very grateful to his advisor Professor Xiaoping Yang for his expert guidances and useful conversations.

\end{document}